\documentclass[12pt]{amsart}

\usepackage[margin=1in]{geometry}

\usepackage{amsmath,amscd,amssymb,amsfonts,latexsym,wasysym, mathrsfs, mathtools,hhline,color}
\usepackage{ulem}

\usepackage[all, cmtip]{xy}
\usepackage{csquotes}
\usepackage{url}
\usepackage{comment,l3keys2e, array, pgfcore}

\usepackage{nicematrix}

\definecolor{hot}{RGB}{65,105,225}

\usepackage[pagebackref=true,colorlinks=true, linkcolor=hot ,  citecolor=hot, urlcolor=hot]{hyperref}
\usepackage{ textcomp }
\usepackage{ tipa }
\usepackage{graphicx,enumerate}

\usepackage{geometry}
\theoremstyle{plain}
\newtheorem{theorem}{Theorem}[section]
\newtheorem{prop}[theorem]{Proposition}

\newtheorem{lm}[theorem]{Lemma}

\newtheorem{cor}[theorem]{Corollary}

\newtheorem{thrm}[theorem]{Theorem}

\theoremstyle{definition}

\newtheorem{defn}[theorem]{Definition}
\newtheorem{que}[theorem]{Question}
\newtheorem{rmk}[theorem]{Remark}

\newtheorem{ex}[theorem]{Example}
\newtheorem*{ex*}{Example}

\def\be{\begin{equation}}
\def\ee{\end{equation}}

\def\bt{\begin{thrm}}
\def\et{\end{thrm}}

\def\bc{\begin{cor}}
\def\ec{\end{cor}}

\def\br{\begin{rmk}}
\def\er{\end{rmk}}

\def\bp{\begin{prop}}
\def\ep{\end{prop}}

\def\bl{\begin{lm}}
\def\el{\end{lm}}

\def\bex{\begin{ex}}
\def\eex{\end{ex}}

\def\bd{\begin{defn}}
\def\ed{\end{defn}}

\newcommand{\C}{\mathbb{C}}
\newcommand{\CP}{\mathbb{CP}}

\newcommand{\Z}{\mathbb{Z}}

\newcommand{\K}{\mathbb{K}}

\newcommand{\B}{\mathbb{B}}

\newcommand{\U}{\mathcal{U}}

\newcommand{\sA}{\mathcal{A}}

\newcommand{\rk}{\mathrm{rank}}

\begin{document}

\title[]{The homology groups of finite cyclic coverings of line arrangement complements}


\author{Yongqiang Liu}
\address{The Institute of Geometry and Physics, University of Science and Technology of China, 96 Jinzhai Road, Hefei, Anhui 230026 China}
\email{liuyq@ustc.edu.cn}

\author{Wentao Xie}
\address{School of Mathematical Sciences, University of Science and Technology of China, 96 Jinzhai Road, Hefei, Anhui 230026 China}
\email{xwt@mail.ustc.edu.cn}

\date{\today}

\keywords{Hyperplane arrangement, Milnor fiber, finite cyclic covering}

\subjclass[2020]{Primary 52C35, 32S55}

\begin{abstract}
In this paper, we study the first homology group of finite cyclic covering of complex line arrangement complement. We show that this first integral homology group is torsion-free under certain condition similar to the one used by Cohen-Dimca-Orlik \cite[Theorem 1]{CDO03}.  In particular, this includes the case of the Milnor fiber, which generalizes the previous results obtained by Williams \cite[Main Theorem 1]{Wil13} for complexified line arrangement to any complex line arrangement.
\end{abstract}

\maketitle


\section{Introduction}

\subsection{Background}
An arrangement of hyperplanes $\sA$ is a finite collection of hyperplanes in a finite dimensional complex vector (affine, or projective) space.  A fundamental problem in the theory of hyperplane arrangements is to decide whether various topological invariants of the complement of $\sA$ are determined by the combinatorial data (i.e., the poset of intersections of hyperplanes) of $\sA$. For instance, the homology groups and the cohomology ring of the hyperplane arrangement complement are combinatorially determined (e.g., see \cite{OT}), while the fundamental group of the complement are not \cite{Ryb11}.   It is an open question whether
 the homology groups of  a finite cover  of a hyperplane arrangement
complement are combinatorially determined. This includes the Milnor fiber of a central hyperplane arrangement. 
See Papadima-Suciu's work \cite{PS17} for recent progress in this direction.   The readers may refer to  Suciu's two survey papers for the required background on hyperplane arrangements:  \cite{Suc01} for an overview of the theory and  \cite{Suc14},  which deals specifically with the boundary manifold and the Milnor fiber of an arrangement.

It is well-known that the homology groups of hyperplane arrangement complement are torsion-free. In fact, Dimca-Papadima \cite{DP03} and Randell \cite{Ran02} independently showed that the complement of a hyperplane arrangement is homotopy equivalent to a minimal CW complex. On the other hand,  Cohen-Denham-Suciu \cite[Theorem 1]{CDS03} gave examples of arrangements with multiplicities whose first homology group of the
Milnor fiber has torsion. Later on Denham-Suciu \cite[Theorem 1.1]{DS14} showed that for every prime $p\geq2$, there are examples of hyperplane arrangements whose higher degree homology of the Milnor fiber has $p$-torsion. This nice result resolved Dimca-N\'{e}methi \cite[Question 7.3]{DN04} and Randell's question \cite[Problem 7]{Ran11}  whether the Milnor fiber of central hyperplane arrangement always has torsion-free homology.  
  In particular, this gives examples of hyperplane arrangement whose Milnor fiber do not necessarily have a minimal CW structure. Yoshinaga \cite[Theorem 1.2]{Yos20} gave the first example of hyperplane arrangement whose first homology group of Minor fiber  has $2$-torsion.   
  
  Denham-Suciu raised the following question:
\begin{que}\label{que}\cite[Question 1.2]{DS14} Is the torsion in the homology of the Milnor fiber of a hyperplane
arrangement combinatorially determined?
\end{que}
Note that the Milnor fiber of central hyperplane arrangement is a finite cyclic covering of the complement of the associated projective hyperplane arrangement. 
One can ask the same question for any finite cyclic covering of hyperplane arrangement complement.
In this paper, we focus on the line arrangement case regarding this question.

\subsection{Main results}

 Let $\sA$ be an essential  line arrangement of $n$ lines in $\CP^{2}$ with complement $\U$. 
Let  $Q(\sA)\in \C[z_0, z_1,z_2]$ be the reduced defining homogeneous  polynomial for $\sA$, defined up to multiplication by a non-zero scalar. The Milnor fiber $F$ of the projective arrangement $\sA$ is defined by $Q(\sA)=1$, which is a smooth hypersurface in $\C^3$. Then $F$ is a $n$-fold cyclic covering of $\U$.
We recall the following result due to Williams, which gives a partial answer to Question \ref{que}.
\bt\label{theorem williams} \cite[Main Theorem 1]{Wil13}  Let $\sA$ be a complexified real arrangement of $n$ lines in $\CP^2$ with the Milnor fiber $F$.  Taking any line $H\in\sA$, let  $\{P_1,\dots,P_s\}$ denote the set of multiple points on $H$ and let $m_i$ denote  the  multiplicity of $P_i$ for $1\leq i\leq s$.  
Then, for any field $\K$, the first Betti number of F with respect to $\K$ satisfies
\be  \label{inequality Williams}
 b_1(F,\K)  \leq n-1+\sum^s_{i=1}(m_i-2)(\gcd(m_i,n)-1).
\ee
Furthermore, if $\gcd(m_i,n)=1$ for every $m_i>2$ (e.g. $n$ is a prime number), then $H_1(F,\Z)\cong \Z^{n-1}$ is torsion-free. 
\et


For a full generalization beyond the Milnor fiber case, we use the following setup. Let $\sA$ be an essential complex arrangement of $n$ lines in $\CP^2$ with complement $\U$. Fix any line $H\in \sA$,  let  $\{P_1,\dots,P_s\}$ denote the set of multiple points on $H$ and let $m_i$ denote  the  multiplicity of $P_i$ for $1\leq i\leq s$.  Here since $\sA$ is essential, we have $s>1$.
We label the other lines according to their intersections with $H$ by two indices: the
first indicating the intersection point and the second indicating the position of the hyperplane with respect to the intersection point. Note that the multiplicity of each intersection point $P_i$ is given by $m_i$, and that $H$ is always the first hyperplane with respect to each intersection point. We then have
$$ \sA = \{H, H_{1,2}, H_{1,3}, \dots, H_{1,m_1}, H_{2,2},\dots, H_{s,m_s} \}.$$
Fix an epimorphism $\varphi:\pi_1(\U)\twoheadrightarrow \Z$ with \begin{center}
$\varphi(\alpha)=\epsilon\in \Z$ and   $\varphi(\alpha_{i,j})=\epsilon_{i,j}\in \Z$ for all $1\leq i\leq s, 2\leq j\leq m_i.$
\end{center} 
Here $\alpha$ and $\alpha_{i,j}$ are the meridians associated to $H$ and $H_{i,j}$ respectively. Set $\varepsilon_i=\epsilon+\sum^{m_i}_{j=2} \epsilon_{i,j}$ for all $1\leq i \leq s$.

\bt \label{main theorem}
 With the above notations and assumptions,   consider the composed map $\pi_1(\U)\overset{\varphi}{ \twoheadrightarrow} \Z \twoheadrightarrow \Z_N$ and let $\U^{\varphi, N}$ denote the associated $N$-fold covering.
If  $\epsilon=1$ and $\varepsilon_i\neq 0$ for all $m_i>2$, then 
for any field $\K$ we have
\be  \label{main inequality}
 b_1 (\U^{\varphi, N},\K) \leq (n-1)+\sum^s_{i=1}(m_i-2)(\gcd(\varepsilon_i,N)-1) .
\ee 
Furthermore, if  $\gcd(\varepsilon_i,N)=1$ for all $m_i>2$, 
 $H_1(\U^{\varphi, N},\Z)\cong \Z^{n-1}$ is torsion-free. 
\et
\br \label{rem assumption} The setup we used here is to consider all finite cyclic covering determined by $\varphi$ at once. This setup is convenient for our proof and can be used to prove a divisibility result for the Alexander polynomial of $\U$ (see Proposition \ref{prop Alexander}), which partially generalizes  
Elduque's result \cite[Theorem 6]{Eld21}. If one only focuses on the $N$-fold covering, the assumption can be changed to 
\begin{center}
     $\epsilon \equiv 1 \mod N $ and $\varepsilon_i\neq 0 \mod N$ for all $m_i>2$.
\end{center}
When $N=2$, this corresponds to the CDO-condition considered by Sugawara in \cite[Definition 2.1]{Sug23}. In particular, it is already proved by the first author joint with Liu \cite[Corollary 1.7]{LL23} that the 2-fold covering of the complement has torsion-free homology groups using the results proved by Sugawara for complexified real arrangement case \cite[Theorem 1.3]{Sug23} and the first author joint with Maxim and Wang \cite[Theorem 1.2]{LMW}.

In fact, the two conditions $\epsilon\neq 0$ and $\varepsilon_i\neq 0$ for all $m_i>2$ are corresponding to the conditions used by Cohen-Dimca-Orlik in \cite[Theorem 1]{CDO03}. We have the condition $\epsilon=1$ in Theorem \ref{main theorem}  to make the computation manageable. 
\er
 
The following result is a direct application of Theorem \ref{main theorem}.
\begin{cor} With the same assumptions and notations as in Theorem \ref{main theorem}, 
if $m_i=2$ for every  multiple point $P_i$ in $H$ and $\epsilon=1$,  $H_1(\U^{\varphi, N},\Z)\cong \Z^{n-1}$ is torsion-free.  
\end{cor}

Since the Milnor fiber is a special case satisfying the assumptions in Theorem \ref{main theorem}, we extend Theorem \ref{theorem williams} to any complex line arrangement.
\bc \label{cor Milnor fiber} Theorem \ref{theorem williams} holds for any complex line arrangement.
\ec 

As noted by Williams in \cite[Section 4]{Wil13}, the inequality (\ref{inequality Williams}) is quite similar to the one obtained by Cohen-Dimca-Orlik in \cite[Theorem 13]{CDO03}, which is linked to the direction of divisibility results for Alexander polynomials of hypersurface complements. This topic was initiated by Libgober \cite{L82,L94} and developed by many other authors, see \cite{DL,Max,DM,CAF,Liu}.
In particular, the recent work by Elduque \cite{Eld21} gives a divisibility result for twisted Alexander polynomial of line arrangement complement. She constructed a tubular neighborhood along one line in loc. cit., which  
plays a key role in the proof of Theorem \ref{main theorem}. The key point in the proof is to compute the fundamental group of the boundary manifold of this neighbourhood.  The computation is done by a method due to Hirzebruch \cite{H}. For related works,  Westlund \cite{W}  computed the fundamental group of the boundary manifold of the complement $\U$. Cohen-Suciu \cite{CS} found a minimal presentation for such fundamental group and computed their twisted Alexander polynomials. Hironaka \cite{Hir} and Florens-Guerville-Marco \cite{FGM}  studied relationships between the fundamental group of a line arrangement complement and that of the boundary manifold of such an arrangement. The computation in our case is similar but much simpler, since our boundary manifold is only part of  the boundary manifold of the line arrangement complement $\U$. 
On the other hand, the multivariable Alexander polynomial of $\U$ was first studied by Dimca-Papadima-Suciu in \cite{DPS08}. In particular, they showed that the multivariable Alexander polynomial of $\U$ is non-constant if and only if the line arrangement is a pencil with at least three lines, see \cite[Proposition 3.4]{DPS08}.

\subsection{Summary}
The paper is organized as follows.
In section 2, we recall the construction of the boundary manifold of the tubular neighbourhood of $H$ and compute its fundamental group. Then we use Fox calculus to compute the Alexander matrix. 
Section 3 is devoted to the proof of Theorem \ref{main theorem}, Corollary \ref{cor Milnor fiber} and the divisibility results we mentioned earlier. 
 


\bigskip

{\bf Convention.} We will use the following convention throughout the paper:
\begin{itemize}
    \item $\Z_N$ denotes the cyclic group  $\Z/N\Z$.
    \item  For a space $X$, the $i$-th Betti number with field coefficients $\K$ of $X$ is denoted by $b_i(X,\K)$.
\item  A covering always means an unbranched covering.
\item We will assume that the base point $x_0 \in X$ is fixed when we consider the fundamental group $\pi_1(X) = \pi_1(X, x_0)$.
\end{itemize}

\bigskip

{\bf Acknowledgements.} Both authors are  supported by National Key Research and Development Project SQ2020YFA070080, NSFC grant No. 12001511, the Project of Stable Support for Youth Team in Basic Research Field CAS (YSBR-001), the starting grant from University of Science and Technology of China, the project ``Analysis and Geometry on Bundles" of Ministry of Science and Technology of the People's Republic of China and  Fundamental Research Funds for the Central Universities.

\section{The boundary manifold of the tubular neighbourhood}
In this section, we focus on the boundary manifold of the tubular neighbourhood along a line $H\in \sA$. 



\subsection{Construction}
We follow Elduque's construction presented in \cite[Section 2, 3]{Eld21}. 

Let $\sA$ be a line arrangement in $\CP^2$ with $n$ lines. 
For a fixed line $H\in \sA$, recall that $H$ has the set of  multiple points $\{P_1,\dots, P_s\}$ with the  multiplicity $m_i$ for $P_i$ and
$$\sA=\{H,H_{1,2},\dots,H_{1,m_1},\dots,H_{s,2},\dots,H_{s,m_s} \}.$$ 

For each singular point $P_i$, we take a  small enough ball $\B_i$ centered at the point $P_i$. 
Then we take a tubular neighbourhood of $H\setminus\cup_{i=1}^s \B_i(\frac{1}{2})$  with a even smaller radius, denoted by $\Sigma$. Here $\B_i(\frac{1}{2})$ is the ball  with the same center as  $\B_i$  but half radius. 
Set $$W=\Sigma \cup (\cup_{i=1}^s \B_i).$$ By rounding corners, its boundary manifold $X=\partial W$  is a real 3-dimensional manifold, which can be described as follows.
Set $N=\big(H\setminus \cup_{i=1}^s \B_i\big)\times S^1$, which should be thought of as the boundary of $\Sigma$ around the non-singular part of $H$. 
We have that $$\partial N= \partial \big(H \setminus\cup_{i=1}^s \B_i\big)\times S^1.$$ Since $\partial(H\setminus \bigcup_{i=1}^s \B_i)$ is a union of disjoint $S^1$'s (one from every disk $H\cap \B_i$ removed), then $\partial N$ is a union of disjoint tori $T_{i}$. 
Let $L_i$ be the link of the singularity at the point $P_i$ (which is a Hopf link with $m_i$ components), and let $S_i^3$ be the boundary of $\B_i$. Consider the space
$$X= N\cup_{\sqcup _i T_i }  (\bigsqcup_{i=1}^s S^3_i \setminus L_i) ,$$
where the gluing is done as follows. 
Note that  $L_i=\bigcup_{j=1}^{m_i} L_{i,j}$, where $L_{i,j}$ is the Hopf link which corresponds to $H_{i,j}$ (here we regard $H$ as $H_{i,1}$).  A meridian around $L_{i,1}$ is glued to $\text{\{point\}}\times S^1 \subset  N$ and $L_{i,1}$ is glued to the $S^1$ corresponding to the boundary of $H \cap \B_i$ for all $1\leq i\leq s$.

\begin{rmk}
One may also consider the boundary manifold of the tubular neighborhood of all lines in $\sA$, denoted by $\partial \U$, since it can be viewed as the boundary manifold of $\U$. Clearly by construction $X$ is only part of $\partial \U$. 
The boundary manifold $\partial \U$ has been thoroughly studied by Westlund in \cite{W} and later by Cohen-Suciu in \cite{CS}. 
\end{rmk}

\subsection{The fundamental group of $X$}
In this subsection, we compute $\pi_1(X)$.
 
We first recall the fundamental group of the local complement $\U\cap \B_i$, which is homotopy equivalent to $\U_i= \U\cap S^3_i$. Let $\alpha_{i,j}$ denotes the meridian associated to $H_{i,j}$ and $\beta_i$ denote the generator of the  fundamental group of the fiber for the Hopf map over $\U_i$. Then we have that 
\begin{equation}\label{eq1}
\begin{aligned}
\pi_1(\U_i) &\cong \left\langle \beta_i,\alpha_{i,1}, \alpha_{i,2},\dots,\alpha_{i,m_i} | \beta_i=\alpha_{i,1}\alpha_{i,2}\cdots \alpha_{i,m_i}, [\beta_i,\alpha_{i,j}] , 1\leq j\leq m_i \right\rangle \\
&\cong \left\langle \beta_i,\alpha_{i,1},\alpha_{i,2},\dots,\alpha_{i,m_i-1} | [\beta_i,\alpha_{i,j}], 1\leq j\leq m_i-1 \right\rangle
\end{aligned}
\end{equation}

Abusing notations, we will look at  the $\alpha_{i,j}$'s and $\beta_i$'s  inside of $\pi_1(X)$ via the maps $\pi_1(\U_i)\to \pi_1(X)$ induced by inclusion.
Since we regard $H$ as $H_{i,1}$, hence $\alpha_{i,1}=\alpha$ for all $1\leq i \leq s$, where $\alpha$ is the meridian along $H$.  In fact,  all $\alpha_{i,1}$'s represent the same element in $\pi_1(X)$ up to homotopy equivalence, since moving the meridian of $H$ along $H\setminus \cup_{i=1}^s \B_i$ does not change its homotopy class. 

\begin{prop}
With the above assumptions and notations, we have \begin{equation}\label{eqX}
\begin{aligned}
\pi_1(X)
& \cong
\left\langle\begin{array}{cc|cc}
 \alpha, &  & \alpha^{1-s}\beta_1\cdots\beta_s=1, & \\
\beta_i, & 1\leq i\leq s,   & \beta_i=\alpha\alpha_{i,2}\cdots \alpha_{1,m_i}, & 1\leq i\leq s\\
\alpha_{i,j},&  2\leq j\leq m_{i} & [\alpha_{i,j},\beta_i] ,[\alpha,\beta_i], & 2\leq j\leq m_{i}
\end{array}\right\rangle 
\\
&\cong 
\left\langle\begin{array}{cc|cc}
\alpha,\beta_i, & 1\leq i\leq s & \alpha^{1-s}\beta_1\cdots\beta_s=1 ,& 1\leq i\leq s\\
\alpha_{i,j}, & 2\leq j\leq m_i-1 & [\alpha_{i,j},\beta_i] ,[\alpha,\beta_i], & 2\leq j\leq m_i-1
\end{array}\right\rangle \\
\end{aligned}
\end{equation}

\end{prop}

\begin{proof}
We use Van-Kampen Theorem to compute $\pi_1(X)$. Note that $S^3_i \setminus L_i$ is isomorphic to $S^1\times (\CP^1 - \{ m_i \text{ points}\})$, which is a trivial bundle. So $X$ is constructed by pasting certain $S^1$-bundles along their boundaries, which are some disjoint tori. 
Using Van-Kampen Theorem, we see that $\pi_1(X)$ can be generated by the elements $\alpha,\beta_i,\alpha_{i,j}$. We have the following commutation relation due to the fundamental group of $\U_i$:
$$[\alpha,\beta_i]=1,[\alpha_{i,j},\beta_i]=1, 1\leq i\leq s, 2 \leq j\leq m_i.
$$ 

To have a better understanding of  $\pi_1(X)$, we follow the approach by Hirzebruch \cite{H}. Consider the blow up of $\CP^2$ at all points $P_i$ for $1\leq i \leq s$, denoted by $\widetilde{\CP^2}$. Let $\widetilde{H}$ denote the strict transform of $H$ in $\widetilde{\CP^2}$.
Consider the $S^1$-bundle over $\widetilde{H}$ associated to the conormal bundle of $\widetilde{H}$ in $\widetilde{\CP^2}$. Since $\widetilde{H}=\CP^1$, this is a $S^1$-bundle over $\CP^1$ with Euler number $1-s$. For a detailed computation of this Euler number, see \cite[Section 2]{H} and \cite[Section 2.2]{W}. 

Since the blow-up map does not change $\U$, we can consider $X\subset \U \subset \widetilde{\CP^2}$.
The left relation relies on the Euler number of $ S^1$  bundle over $\widetilde{H}$. In fact, considering the product of the generators (total boundary) in the $ S^1$  bundle over $\widetilde{H}$ gives the following relation
$\alpha^{1-s}\beta_1\cdots\beta_s=1.$
For a detailed discussion, see  \cite[Section 2]{H} by Hirzebruch and \cite[Section 2.2.2]{W} by Westlund.
\end{proof}

\br Using the commutation relations, $\alpha^{1-s}\beta_1\cdots\beta_s=1$ is equivalent to the relation $$\alpha \alpha_{1,2}\cdots \alpha_{1,m_1}\cdots \alpha_{s,2}\cdots \alpha_{s,m_s}=1.$$ This phenomenon is already explored by Cohen-Suciu \cite[Section 3.6]{CS}. It is compatible with the fact that the map $\pi_1(X)\to \pi_1(\U)$ induced by inclusion is an epimorphism (see section 3.1) and the same relation holds for $\pi_1(\U)$, see e.g. \cite[Corollary 4.7 (ii)]{Dim17}.
\er

\subsection{Fox calculus and the Alexander matrix}
In this subsection, we use Fox calculus to compute the Alexander matrix. 
We first briefly recall the basic properties of Fox calculus. 
For more details, see \cite{Fox56} and \cite[Chapter \uppercase\expandafter{\romannumeral7}]{CF77}. 

Given a finitely presented group $G=\left\langle x_1,\dots,x_m|r_1,\dots,r_l\right\rangle$, let $\tilde{G}=\left\langle x_1,\dots, x_m\right\rangle$ be a free group with $m$-generators. Let $\xi:\tilde{G}\xrightarrow{} G$  
and $\rho\colon G\xrightarrow{} \dfrac{G}{[G,G]}$ denote  the natural quotient morphisms, where $[G,G]$ is the commutator subgroup of $G$.  Note that $\rho$ can be uniquely  extended to a homomorphism on the corresponding group rings, still denoted by $\rho$ by abuse of notations.

Define the $i$-th Fox derivative  $\dfrac{\partial}{\partial x_i}: \tilde{G} \rightarrow \Z G$ for all $1\leq i \leq m$ such that 
$$\dfrac{\partial}{\partial x_i}(x_j)=\delta_{ij}$$ and
$$\dfrac{\partial}{\partial x_i}(g_1g_2)=\dfrac{\partial}{\partial x_i}(g_1)+ \xi(g_1)\dfrac{\partial}{\partial x_i}(g_2),$$
where $\delta_{ij}$ is Kronecker symbol, $g_1,g_2$ are elements of $\tilde{G}$ .

\bd The Alexander matrix of group $G$ is a matrix $M=(c_{ij})$, where the size of $M$ is $l\times m$ and the entries are given by $c_{ij}=\rho (\dfrac{\partial r_i}{\partial x_j})$.  
\ed

Now let us do the Fox calculus for $\pi_1(X)$. For the relations in $\pi_1(X)$, we have the  following formulas:
$$\dfrac{\partial}{\partial \alpha}(\alpha^{-1})=-\alpha^{-1} ,$$

$$\dfrac{\partial}{\partial \alpha}(\alpha^k)=\dfrac{\alpha^k-1}{\alpha-1},  $$

$$\dfrac{\partial}{\partial \alpha_{i,j}}( [\alpha_{i,j},\beta_i])
=1-\alpha_{i,j}\beta_i \alpha_{i,j}^{-1}
\overset{\rho}{\mapsto} 1-\beta_i,
$$

$$\dfrac{\partial }{\partial \beta_i}( [\alpha_{i,j},\beta_i])
=\alpha_{i,j}-\alpha_{i,j}\beta_i \alpha_{i,j}^{-1} \beta_i^{-1}
\overset{\rho}{\mapsto} \alpha_{i,j}-1,
$$

$$\dfrac{\partial }{\partial \beta_i}(\alpha ^{1-s}\beta_1\cdots \beta_i)
=\alpha^{1-s}\beta_1\cdots \beta_{i-1}
\overset{\rho}{\mapsto} \alpha^{1-s}\beta_1\cdots \beta_{i-1}.
$$
Then we get the Alexander matrix of $\pi_1(X)$:
\begin{equation}
M=
\begin{pNiceMatrix}
\dfrac{\alpha^{(1-s)}-1}{\alpha-1}       & Q_1  & Q_2  & \cdots & Q_s  \\
V_1    & M_1  &      &        &      \\
V_2    &      & M_2  &        &      \\
\vdots &      &      & \ddots &      \\
V_s    &      &      &        & M_s

\end{pNiceMatrix},
\end{equation}
where the empty entries are zero and $M_i,V_i,Q_i$ are matrices with the following form:  
\begin{equation}
\begin{aligned}
M_i&=
\begin{pNiceMatrix}[first-row,first-col]
& \beta_i & \alpha_{i,2} & \cdots & \alpha_{i,m_i-1} \\
[\alpha,\beta_i] &  \alpha-1  \\
[\alpha_{i,2},\beta_i]  &\alpha_{i,2}-1 & 1-\beta_i \\ 
\vdots &  \vdots &  & \ddots\\
[\alpha_{i,m_i-1},\beta_i]  &\alpha_{i,m_i-1}-1 & & & 1-\beta_i\\   
\end{pNiceMatrix}_{(m_i-1)\times (m_i-1)}, \\
V_i&=(1-\beta_i,0,\cdots,0)^{\mathsf{T}}_{1\times (m_i-1)} , \\
Q_i&=(\alpha^{1-s}\beta_1\cdots \beta_{i-1},0,\cdots,0)_{1\times (m_i-1)}.
\end{aligned}
\end{equation}
Here  $\mathsf{T}$ denote the transpose of  the matrix.

\section{Proof of the main results}
In this section, we give the proof of Theorem \ref{main theorem}.
First we list two equalities that will be used later:
\be \label{equality 2} \epsilon + \sum_{i=1}^s \sum_{j=2}^{m_i} \epsilon_{i,j}= (1-s)\epsilon +\sum_{i=1}^s \varepsilon_i=0\ee
and
\be \label{equality 1} n=1 +\sum_{i=1}^s (m_i-1)=1-s+\sum_{i=1}^{s} m_i. \ee 
\subsection{Divisibility results for Alexander polynomial}
Consider the inclusion map $X\to \U$. Note that $X$ is connected, hence the map induced by inclusion $\pi_0(X)\to \pi_0(\U)$ is an isomorphism. By a Lefschetz hyperplane section theorem type argument, the map  $\pi_1(X)\to \pi_1(\U)$ induced by inclusion is an epimorphism. In fact, we can take a generic line $L$ in $\CP^2$ with respect to $\sA$, which is close enough to the line $H$ such that $L$ is contained in the tubular neighborhood we constructed in section 2. Then  
the map $\pi_1(L\cap \U)\twoheadrightarrow \pi_1(\U)$ is surjective by the Lefschetz hyperplane section theorem (see \cite[Proposition 4.3.1]{D1}).   Note that $L\cap \U \subset  W\setminus H$ and $W\setminus H$ is homotopy equivalent to $X$. Putting all together, we get that the map $\pi_1(X)\to \pi_1(\U)$ induced by inclusion is an epimorphism.
 
 Consider the composed maps $\pi_1(X)\to \pi_1(\U) \overset{\varphi}{\to} \Z$. 
 Let $X^\varphi$ denote the corresponding infinite cyclic covering space and $X^{\varphi,N}$ denote the $N$-fold cyclic covering for any integer $N>1$ induced by further composing with $\Z \to \Z_N$. Then we have the following result simply due to the fact that $\pi_1(X)\to \pi_1(\U)$ is an epimorphism. 
\bl \label{lem epimorphism}
With the above assumptions and notations,   the induced map $H_1(X^{\varphi,N},\K)\to H_1(\U^{\varphi,N},\K)$ is surjective for any field $\K$.
\el

For any field $\K$, by lifting the cell structure of $X$ to $X^\varphi$,
we obtain a free basis for the cellular chain complex of $X^\varphi$ as $\K[t^{\pm}]$-modules. 
Then the cellular chain complex $C_*(X^\varphi,\K)$ is a bounded complex of finitely generated free $\K[t^{\pm}]$-modules.
The first Alexander module $H_1(X^\varphi,\K)$ is the first homology of the complex   $C_*(X^\varphi,\K)$, hence $H_1(X^\varphi,\K)$  is a finitely generated $\K[t^{\pm}]$-modules. Note that $\K[t^{\pm}]$ is a PID.

\bd For any field $\K$, the first Alexander polynomial of $X$ with respect to the covering space $X^\varphi $ is defined 
as the order of the torsion part of $H_1(X^\varphi,\K)$, denoted by $\Delta_1(X^\varphi,\K)$. 
\ed

Consider the covering space $\U^\varphi$ associated to the map $\pi_1(\U) \overset{\varphi}{\to} \Z$.
For any field $\K$, the first Alexander polynomial $\Delta_1(\U^\varphi,\K)$ of $\U$ with respect to the covering space $\U^\varphi$ can be defined similarly.

\bp \label{prop Alexander}
With the above assumptions and notations, if $\epsilon\neq 0$ and $\varepsilon_i\neq 0$ for all $m_i>2$,  then for any field $\K$, $H_1(\U^\varphi,\K)$ is a torsion $\K[t^{\pm}]$-module and 
  $$\Delta_1(\U^\varphi,\K) \mid \ (t-1)\cdot(t^{\epsilon}-1)^{s-2} 
\cdot \prod^s_{i=1}(t^{\varepsilon_i}-1)^{m_i-2}.$$ 
\ep
\begin{proof} 
We use Fox calculus to compute the first Alexander polynomial $\Delta_1(X^\varphi,\K)$. 
In fact, it can be computed by the following Alexander matrix associated to  $X^\varphi$:
\begin{footnotesize}
\begin{equation}
\begin{pNiceMatrix}[first-row,first-col]
 & \alpha & \beta_1 & \alpha_{1,2} & \cdots & \alpha_{1,m_1-1} & \cdots  &\beta_s & \alpha_{s,2} & \cdots & \alpha_{s,m_s-1}  \\
 \alpha^{1-s}\beta_1 \cdots \beta_s &\dfrac{t^{\epsilon (1-s)}-1}{t^\epsilon-1} & t^{\epsilon (1-s)} & 0 & \cdots & 0 & \cdots &  t^{\epsilon (1-s)+\sum_{i=1}^{s-1} \varepsilon_i} & 0 & \cdots & 0\\
[\alpha,\beta_1] & 1-t^{\varepsilon_1} & t^{\epsilon}-1 \\
[\alpha_{1,2},\beta_1] &0 &t^{\epsilon_{1,2}}-1 & 1-t^{\varepsilon_1}\\ 
\vdots &\vdots & \vdots & & \ddots\\
[\alpha_{1,m_1-1},\beta_1] &0 & t^{\epsilon_{1,m_1-1}}-1 & & & 1-t^{\varepsilon_1}\\  
\vdots &\vdots &  & &  & &\ddots\\
\vdots &\vdots &  & &  & &\\
[\alpha,\beta_s] & 1-t^{\varepsilon_s} & & & & & & t^{\epsilon}-1 \\
[\alpha_{s,2},\beta_s] & 0 & & & & & & t^{\epsilon_{s,2}}-1 & 1-t^{\varepsilon_s} \\ 
\vdots  & \vdots  &  & &  & & & \vdots& & \ddots \\
[\alpha_{s,m_s-1},\beta_s] & 0 & & & & & & t^{\epsilon_{s,m_s-1}}-1 & & & 1-t^{\varepsilon_s}\\  
\end{pNiceMatrix}.
\end{equation}
\end{footnotesize}

First, one can show the determinant of this matrix is 0 by cofactor expansion of the first row. Then 
it is easy to see that this matrix has rank $n-1$ if $\epsilon\neq 0$ and  $\varepsilon_i\neq 0$ for all $m_i>2$ (if $m_i=2$, one can see that the column for $\alpha_{i,2}$ does not appear, hence the value of $\varepsilon_i$ does not effect the rank in this case). 
In particular, $H_1(X^\varphi,\K)$ is a torsion $\K[t^\pm]$-module.

Using some elementary column and row transformations, we get the following new matrix:
\begin{footnotesize}
\begin{equation*}
\begin{pNiceArray}[first-row]{ccccc|ccccccc}
 \alpha & \beta_1 & \beta_2 &\cdots & \beta_s & \alpha_{1,2} & \cdots & \alpha_{1,m_1-1}& \cdots  & \alpha_{s,2} & \cdots & \alpha_{s,m_s-1} \\
\dfrac{t^{\epsilon(1-s)}-1}{t^{\epsilon}-1} &  t^{\epsilon(1-s)} & t^{\epsilon(1-s)+\varepsilon_1} & \cdots & t^{\epsilon(1-s)+\sum_{i=1}^{s-1}\varepsilon_i}\\
1-t^{\varepsilon_1} & t^{\epsilon}-1\\
1-t^{\varepsilon_2} & &t^{\epsilon}-1\\ 
\vdots & & & \ddots \\
1-t^{\varepsilon_s} & & & & t^{\epsilon}-1\\ \hline
& t^{\epsilon_{1,2}}-1 & & & & 1-t^{\varepsilon_1} \\
& \vdots & & & & & \ddots \\ 
& t^{\epsilon_{1,m_1-1}}-1 & & & & & & 1-t^{\varepsilon_1} \\ 
&&&\ddots&&&&&\ddots \\ 
&&&&t^{\epsilon_{s,2}}-1 &&&&& 1-t^{\varepsilon_s} \\
&&&& \vdots  &&&&&& \ddots \\
&&&& t^{\epsilon_{s,m_s-1}}-1 &&&&&&& 1-t^{\varepsilon_s} \\

\end{pNiceArray}.
\end{equation*}
\end{footnotesize}

Set $$f(t)=(t^{\epsilon}-1)^{s-2} \prod^s_{i=1}(t^{\varepsilon_i}-1)^{m_i-2}.$$  We consider some minors with size $n-1$. 
Removing the row for $[\alpha,\beta_i]$ and the first column, the column for $\beta_i$ and the column for $\alpha_{i,j}$ with $1\leq i \leq s$ and $2\leq j \leq m_i-1$ respectively, the corresponding  minors are
\begin{center}
   $(t^{\epsilon}-1)\cdot f(t)$, $(t^{\varepsilon_{i}}-1)\cdot f(t) $ and $(t^{\epsilon_{i,j}}-1)\cdot f(t) $ 
\end{center}
up to a unit of $\K[t^{\pm}]$. For  the computation of these minors, one needs  the following equality
$$0=1-t^{\epsilon(1-s)}+(1-t^{\varepsilon_1})t^{\epsilon(1-s)}+\cdots+(1-t^{\varepsilon_{s}})t^{\epsilon(1-s)+\varepsilon_1+\varepsilon_{2}+\cdots+\varepsilon_{s-1}}, $$
which follows from (\ref{equality 2}).

Since $\varepsilon_i=\epsilon+\sum^{m_i}_{j=2}\epsilon_{i,j}$, we have the following equality
 $$1-t^{\varepsilon_i}=(1-t^{\epsilon})+(1-t^{\epsilon_{i,2}})t^{\epsilon}+\cdots +(1-t^{\epsilon_{i,m_i}})t^{\epsilon+\epsilon_{i,2}+\cdots+\epsilon_{i,m_i-1}} .$$
Putting all these together, we have that the Alexander polynomial of $X^\varphi$ divides 
$$\gcd(t^{\epsilon}-1,t^{\epsilon_{1,2}}-1,\dots, t^{\epsilon_{1,m_i}}-1, t^{\epsilon_{2,2}}-1, \dots, t^{\epsilon_{s,m_s}}-1)f(t).$$ 
Note that $\varphi: \pi_1(\U)\twoheadrightarrow \Z$ is surjective, 
which implies  $\gcd(\epsilon,\epsilon_{1,2},\dots,\epsilon_{s,m_s})=1$. 
This fact implies that 
$$\gcd(t^{\epsilon}-1,t^{\epsilon_{1,2}}-1,\dots, t^{\epsilon_{1,m_i}}-1, t^{\epsilon_{2,2}}-1,\dots, t^{\epsilon_{s,m_s}}-1)=t-1,$$ 
which can be proved by the similar trick for above equality of $1-t^{\varepsilon_i}$. 
Hence the Alexander polynomial of $X^\varphi$ divides $(t-1)f(t)$.  We leave it as an exercise for the readers to  check that the proof still works even if some $\epsilon_{i,j}=0$.

On the other hand, since $\pi_1(X) \to \pi_1(\U)$ is surjective, it follows that the inclusion map induces an epimorphism 
$$H_1(X^\varphi,\K)\twoheadrightarrow H_1(\U^\varphi,\K), $$
see e.g. \cite[Proposition 2]{Eld21}. Putting all together, we get that $H_1(\U^\varphi,\K)$ is a torsion $\K[t^\pm]$-module
and the Alexander polynomial of $\U^{\varphi}$ divides the Alexander polynomial of $X^\varphi$. Then the claim follows. 
\end{proof}
\br This proposition is compatible with \cite[Theorem 1]{CDO03} and \cite[Remark 6.9]{LM} for $\K=\C$. 
One may also compare this proposition with the one obtained by Elduque \cite[Theorem 6]{Eld21}. Elduque studied the boundary manifold for the affine tubular neighborhood which is slightly different from ours. She computed the Alexander polynomial by Reidemeister torsion, while we did it by Fox calculus after computing the fundamental group of $X$. In particular, once can check that Proposition \ref{prop Alexander} implies \cite[Theorem 4, Theorem 6]{Eld21} for the non-twisted version. For example, to have $H_1(\U,\K)$ being a torsion $\K[t^\pm]$-module, \cite[Theorem 4]{Eld21} needs the assumption $\epsilon_{i,j}>0$ for all $1\leq i \leq s$ and $2\leq j \leq m_i$. Using equality (\ref{equality 2}) and $s>1$, this certainly implies the assumption in Proposition \ref{prop Alexander}.

 On the other hand, the twisted multivariable Alexander polynomial of the boundary manifold $\partial \U$ was first studied by Cohen-Suciu in \cite[Section 4, 5]{CS}. 
Note that the inclusion map $\partial \U \to \U$ induces an epimorphism on fundamental groups (see e.g. \cite[Remark 10]{Eld21}): 
$$ \pi_1(\partial \U)\twoheadrightarrow \pi_1(\U) .$$
Consider the composed map $\pi_1(\partial \U) \twoheadrightarrow \pi_1(\U) \overset{\varphi}{\to} \Z.$ Let $(\partial \U)^\varphi$ denote the corresponding infinite cyclic covering space. 
The assumptions in Proposition \ref{prop Alexander} are not enough to show that $H_1((\partial \U)^\varphi,\K)$ is a torsion $\K[t^\pm]$-module. But with stronger assumptions as in \cite[Remark 10]{Eld21}, Elduque showed that $H_1((\partial \U)^\varphi,\K)$  is torsion and the corresponding Alexander polynomial is very similar to the  twisted multivariable Alexander polynomial obtained by Cohen and Suciu in \cite[Theorem 5.2]{CS}. In particular, she showed in  loc. cit.  that
\begin{center}
     $ \Delta_1(X^\varphi,\K)$ divides $ \Delta_1((\partial \U)^\varphi,\K) .$ 
\end{center}
\er

\subsection{Proof of Theorem \ref{main theorem}}
To prove Theorem \ref{main theorem}, we need the following lemma.
\bl \label{lem linear algebra} Consider the integer matrix $C_N$ of size $N\times N$:
$$C_N=
\begin{pNiceMatrix}
0 & 1 & 0 & \cdots & 0\\
0 & 0 & 1 & \cdots & 0\\
\vdots & & & & \vdots\\
0 & 0 & 0 & \cdots & 1\\
1 & 0 & 0 & \cdots & 0\\
\end{pNiceMatrix}.
$$ For any positive integer $k$ and any field $\K$, $\rk_\K \big(C_N^k-I_N\big)=N-\gcd(k,N)$, where $I_N$ is identity matrix. 
\el 
\begin{proof}
Since $C_N^N=I_N$, we can assume that $1\leq k < N$ without loss of generality (the case $k=N$ is clear). 
Since $(C_N^k-I_N)^\mathsf{T}= C_N^{N-k}-I_N$ (here $*^\mathsf{T}$ denotes the transpose matrix), we only focus on the case $N-k\leq k$ in the rest of the proof. Then  $C_N^k-I_N$ equals 
$$
\begin{pNiceArray}[first-row,first-col]{ccc|ccccc}
      & 1 & \cdots      &       &       &  \cdots     & k+1     &  \cdots      &    \\
1     & -1 &       &       &       &       & 1     &       &    \\
\vdots &    &\ddots &       &       &       &       &\ddots &    \\
N-k   &    &       & -1    &       &       &       &       & 1  \\ \hline
N-k+1 & 1  &       &       & -1    &       &       &       &    \\
\vdots &    &\ddots &       &       &\ddots &       &       &    \\
\vdots &    &       &\ddots &       &       &\ddots &       &    \\
\vdots &    &       &       &\ddots &       &       &\ddots &    \\
N     &    &       &       &       & 1     &       &       & -1
\end{pNiceArray}
$$     
where all the entries that are not on the three skew lines are zero.

Using row operations, we can kill the lower-left sub-matrix by adding the $i$-th row to the $(N-k+i)$-th row for all $1\leq i\leq N-k$. Next, using column operations, we can kill the upper-right sub-matrix. 
Then $C_N^k-I_N$ becomes in the form:
$$\begin{pNiceMatrix}
I_{N-k} & 0 \\
0       & C_{k}^{\ell} -I_{k} 
\end{pNiceMatrix}$$  
for some positive integer $\ell$.

Note that $C_{k}^{\ell} -I_{k}$ corresponds to $(t^\ell-1)(C_k)$. Since $k<N$, using induction on $N$  we can prove that the Smith normal form of $C_N^k-I_N$ is a matrix with entries on the diagonal being either $1$ or $0$. The starting step for $N=2$ can be checked easily. 
So the rank of $ (C_N^k-I_N)$ does not depend on the choice of $\K$ and we only need to compute $\rk_\C (C_N^k-I_N)$.

Note that the characteristic polynomial of $C_N$ is $t^N-1$.
Then it is easy to see that $\rk_\C (C_N^k-I_N)=N-\gcd(k,N)$.
\end{proof}

\begin{proof}[Proof of Theorem \ref{main theorem}]
We proceed the proof by two steps.
\medskip 
\item[Step 1:] We first simplify the Alexander matrix of $X$ associated to $X^\varphi$ by elementary transformations.

Recall the following the Alexander matrix $A_\varphi$ associated to $X^\varphi$ by furthermore assuming $\epsilon=1$, where we identify the group ring $\Z[\Z]$ as $\Z[t^\pm]$:
\begin{footnotesize}
$$\begin{pNiceMatrix}[first-row,first-col]
 & \alpha & \beta_1 & \alpha_{1,2} & \cdots & \alpha_{1,m_1-1} & \cdots  &\beta_s & \alpha_{s,2} & \cdots & \alpha_{s,m_s-1}  \\
 \alpha^{1-s}\beta_1 \cdots \beta_s &\dfrac{t^{1-s}-1}{t-1} & t^{ 1-s} & 0 & \cdots & 0 & \cdots &  t^{ (1-s)+\sum_{i=1}^{s-1} \varepsilon_i} & 0 & \cdots & 0\\
[\alpha,\beta_1] & 1-t^{\varepsilon_1} & t-1 \\
[\alpha_{1,2},\beta_1] &0 &t^{\epsilon_{1,2}}-1 & 1-t^{\varepsilon_1}\\ 
\vdots &\vdots & \vdots & & \ddots\\
[\alpha_{1,m_1-1},\beta_1] &0 & t^{\epsilon_{1,m_1-1}}-1 & & & 1-t^{\varepsilon_1}\\  
\vdots &\vdots &  & &  & &\ddots\\
\vdots &\vdots &  & &  & &\\
[\alpha,\beta_s] & 1-t^{\varepsilon_s} & & & & & & t-1 \\
[\alpha_{s,2},\beta_s] & 0 & & & & & & t^{\epsilon_{s,2}}-1 & 1-t^{\varepsilon_s} \\ 
\vdots  & \vdots  &  & &  & & & \vdots& & \ddots \\
[\alpha_{s,m_s-1},\beta_s] & 0 & & & & & & t^{\epsilon_{s,m_s-1}}-1 & & & 1-t^{\varepsilon_s}\\  
\end{pNiceMatrix}.
$$
\end{footnotesize}
 Next we do some elementary transformations for $A_\varphi$. 
First, we add the sum of the column for $\beta_i$ multiplying $\dfrac{t^{\varepsilon_i}-1}{t-1}$ to the first column. 
Then the entry $(1,1)$ becomes to
$$\dfrac{t^{1-s}-1}{t-1}
+\sum^s_{i=1} t^{1-s+\varepsilon_1+\cdots+\varepsilon_{i-1}}
\dfrac{t^{\varepsilon_i}-1}{t-1}
=\dfrac{t^{1-s+\varepsilon_1+\cdots +\varepsilon_s}-1}{t-1}
=\dfrac{t^0-1}{t-1}
=0,
$$
where the second equality follows from (\ref{equality 2}).  

Note that the entry $([\alpha_{i,j},\beta_i],1)$ is $\dfrac{(t^{\varepsilon_i}-1)(t^{\epsilon_{i,j}}-1)}{t-1}$, which can be killed by the entry $t^{\epsilon_i}-1$ in  $\alpha_{i,j}$ column for $2\leq j \leq m_i-1$. 
Here if $\epsilon_{i,j}=0$, this entry automatically equals to $0$, and if $m_i=2$, the row for $[\alpha_{i,j},\beta_i]$ does not appear.  
Then the matrix takes in the following form:

$$\begin{pNiceMatrix}
0 & t^{ 1-s} & 0 & \cdots & 0 & \cdots &t^{ (1-s)+\sum_{i=1}^{s-1} \varepsilon_i}  & 0 & \cdots & 0\\
0  & t-1 \\
0 &  & 1-t^{\varepsilon_1}\\ 
\vdots & & & \ddots\\
0 &  & & & 1-t^{\varepsilon_1}\\  
\vdots &  & &  & &\ddots\\
\vdots & &  & &  & &\\
0 & & & & & & t-1 \\
 0 & & & & & &  & 1-t^{\varepsilon_s} \\ 
\vdots    &  & &  & & & & & \ddots \\
0 & & & & & &  & & & 1-t^{\varepsilon_s}\\  
\end{pNiceMatrix}.$$

Add a multiple with $-(t-1)t^{s-1}$ of the first row to the second row. Then the entry $(2,2)$ becomes 0,  and the other non-zero entries on the new second row can be killed by the row corresponding to $[\alpha,\beta_i]$, whose only non-zero entry is $(t-1)$ on the diagonal. 

Finally by these procedures $A_\varphi$ becomes to a diagonal matrix
\be \label{diagonal} 
A'_\varphi=\mathrm{diag}(0, t^{1-s}, \overbrace{t-1,\dots,t-1}^{s-1}, \overbrace{ 1-t^{\varepsilon_1}}^{m_1-2},\dots,\overbrace{ 1-t^{\varepsilon_s}}^{m_s-2}). 
\ee

\medskip 
\item[Step 2:]  We compute $H_1(X^{\varphi,N},\K)$ for any field $\K$.

Since we focus on the first homology which is determined by $\pi_1(X)$, we can assume that $X$ is a 2-dimensional CW complex with its 2-cells corresponding to the relations in $\pi_1(X)$. 
The group of deck transformations of the covering space $X^\varphi$ is isomorphic to $\Z$ and acts on it. 
By lifting the cell structure of $X$ to $X^\varphi$,
we obtain a free basis for the cellular chain complex of $X^\varphi$ as $\K[t^{\pm}]$-modules. 
Then the cellular chain complex $C_*(X^\varphi,\K)$ is a bounded complex of finitely generated free $\K[t^{\pm}]$-modules:

$$0\rightarrow C_2(X^\varphi,\K) \xrightarrow{\partial_2} C_1(X^\varphi,\K) \xrightarrow{\partial_1} C_0(X^\varphi,\K) \rightarrow 0.$$
With the  right cell structure of $X$ with respect to the presentation of $\pi_1(X)$ in section 2.2, $\partial_2$ can be written down as $A_\varphi$. Note that one can compute $H_1(X^{\varphi,N},\K)$ by taking the first homology (as $\K$-vector spaces) of the chain complex
$C_*(X^\varphi,\K) \otimes_{\K[t^{\pm}]} \K[t^{\pm}]/(t^N-1) $.
The map $\partial_2$ becomes $A_\varphi(C_N)$, which  we replace $t$ by $C_N$. The elementary transformation of $\Z[t,t^{-1}] $-coefficient on $A_\varphi$ corresponds to elementary transformation of block matrix on $A_{\varphi}(C_N)$. So we can work on the diagonal matrix $A'_\varphi$ directly.

By Lemma \ref{lem linear algebra}, we have that 
$$ \begin{aligned}
\rk_\K A_\varphi(C_N)
&=N+(s-1)(N-1)+\sum^s_{i=1}(m_i-2)(N-\gcd(\varepsilon_i,N))\\
&=\left(sN+(N-1)\sum^s_{i=1}(m_i-2)\right)
-s+1-\sum^s_{i=1}(m_i-2)(\gcd(\varepsilon_i,N)-1)\\
&=(n-1)N-n+2-\sum^s_{i=1}(m_i-2)(\gcd(\varepsilon_i,N)-1),
\end{aligned}
$$
where the second equality follows from (\ref{equality 1}).

On the other hand, since the composed map $\pi_1(X)\to \pi_1(\U) \to \Z_N$ is surjective, $X^{\varphi,N}$ is connected. Hence 
$$\dim_\K H_1(X^{\varphi,N},\K)= nN-(N-1)-\rk_\K A_\varphi(C_N)=n-1+\sum^s_{i=1}(m_i-2)(\gcd(\varepsilon_i,N)-1)$$
Then the inequality (\ref{main inequality}) follows from Lemma \ref{lem epimorphism}.

 For any field $\K$, by  the Universal Coefficients Theorem and \cite[Corollary 2.5]{CDS03}  we have
$$\dim_\K H_1(\U^{\varphi,N},\K)\geq \dim_\C H_1(\U^{\varphi,N},\C) \geq \dim_\C H_1(\U,\C) =n-1.$$ Therefore, if  $\gcd(\varepsilon_i,N)=1$ for all $m_i>2$, 
$\dim_\K H_1(\U^{\varphi,N},\K) = n-1$ for any field $\K$, hence $H_1(\U^{\varphi_N},\Z)\cong \Z^{n-1}$ by the Universal Coefficients Theorem.
\end{proof}
\begin{proof}[Proof of Corollary \ref{cor Milnor fiber}]
 Cohen and Suciu first showed in \cite[Proposition 1.2]{CS95} that the  Milnor fiber is the finite cyclic cover of $U$  associated to the map $\pi_1(\U) \to \Z_n$, which sends every meridian of lines in $\sA$ to $1\in \Z_n$. See also \cite[Theorem 4.10]{Suc14} for a different proof. 
  By Remark \ref{rem assumption}, the Minor fiber satisfies the assumption in Theorem \ref{main theorem} and the claim follows.   
\end{proof}

\bex Consider the icosidodecahedral arrangement consisting of 16 lines in $\CP^2$ introduced by Yoshinaga \cite[Fig.3]{Yos20}.  
Every line in this arrangement has only multiple points with multiplicity 2 or 4. In particular, the line $H_{16}$ has exactly 5 multiple points with multiplicity all being 4. So torsion in the homology of the Milnor fiber is possible according to Theorem \ref{main theorem}. In fact, Yoshinaga showed that $H_1(F,\Z)$ has 2-torsion \cite[Theorem 1.2]{Yos20}.

On the other hand, we have a lot of 16-fold covering of $\U$, whose first homology group is torsion-free by Theorem \ref{main theorem}. For example, 
consider the 16-fold covering by sending $\alpha_{1}\to 1$, $\alpha_2,\alpha_5,\alpha_8 \to 2$, $\alpha_3 \to -2$ and other lines to $1$. 
Then the first integral homology group of this 16-fold cover is torsion-free by applying Theorem \ref{main theorem} on the first line $H_1\in \sA$ as in \cite[Fig.3]{Yos20}. 
\eex

\bibliographystyle{amsalpha}

\begin{thebibliography}{ADMSP}
\bibitem[CF07]{CAF} J. I. Cogolludo Agust\'{\i}n, V. Florens, {\it Twisted Alexander polynomials of plane algebraic curves}, J. Lond. Math. Soc. (2) 76 (2007), no. 1, 105-121.

\bibitem[CDS03]{CDS03} D. C. Cohen, G. Denham, A. I. Suciu, {\it Torsion in Milnor fiber homology},  Algebr. Geom. Topol. 3 (2003), 511-535.

\bibitem[CDO03]{CDO03} D. C. Cohen, A. Dimca, P. Orlik, {\it Nonresonance conditions for arrangements}, Ann. Inst. Fourier (Grenoble) 53 (2003), no. 6, 1883-1896.

\bibitem[CS95]{CS95} D. C. Cohen, A. I. Suciu, {\it On Milnor fibrations of arrangements}, J. London Math. Soc. (2) 51 (1995), no.~1, 105-119.


\bibitem[CS08]{CS} D. C. Cohen, A. I. Suciu, {\it The boundary manifold of a complex line arrangement}, Groups, homotopy and configuration spaces, Geom. Topol. Monogr., vol. 13, Geom. Topol. Publ., Coventry, 2008, pp. 105-146.


\bibitem[CF77]{CF77}
R. H. Crowell and R. H. Fox, \textit{Introduction to knot theory},  vol. No. 57., Springer-Verlag, New York-Heidelberg, 1977. 

\bibitem[DS14]{DS14} G. Denham, A. Suciu, {\it  Multinets, parallel connections, and Milnor fibrations of arrangements}, Proc. Lond. Math. Soc. (3) 108 (2014), no. 6, 1435-1470.


\bibitem[Dim92]{D1}
A. Dimca, \textit{Singularities and topology of hypersurfaces}, Universitext, Springer-Verlag, New York, 1992. 


\bibitem[Dim17]{Dim17} A. Dimca, {\it Hyperplane arrangements. An introduction.} Universitext. Springer, Cham, 2017.


\bibitem[DL06]{DL} A. Dimca, A. Libgober, \textit{Regular functions transversal at infinity}, Tohoku Math. J. (2) 58 (2006), no. 4, 549-564.

\bibitem[DM07]{DM} A. Dimca,  L. Maxim,
{\it Multivariable Alexander invariants of hypersurface complements}, Trans. Amer. Math. Soc.  359 (2007), no. 7, 3505-3528.

\bibitem[DN04]{DN04}  A. Dimca, A. N\'{e}methi, {\it Hypersurface complements, Alexander modules and monodromy}, Real and complex singularities, Contemp. Math., vol. 354, Amer. Math. Soc., Providence, RI, 2004, pp. 19-43.


\bibitem[DP03]{DP03} A. Dimca, S. Papadima, {\it Hypersurface complements, Milnor fibers and higher homotopy groups of arrangements},  Ann. of Math. (2) 158 (2003), no. 2, 473-507.

\bibitem[DPS08]{DPS08} A. Dimca,  S. Papadima and A. I. Suciu, {\it Alexander polynomials: essential variables and multiplicities}, Int. Math. Res. Not. IMRN 2008, no.~3, Art. ID rnm119, 36 pp.



\bibitem[Eld21]{Eld21} E. Elduque, {\it Twisted Alexander modules of hyperplane arrangement complements},  Rev. R. Acad. Cienc. Exactas F\'{\i}s. Nat. Ser. A Mat. RACSAM 115 (2021), no. 2, Paper No. 70, 28 pp.

\bibitem[FGM15]{FGM}  V. Florens, B. Guerville-Ball\'{e}, M. A. Marco-Buzunariz, {\it On complex line arrangements and their boundary manifolds}, Math. Proc. Cambridge Philos. Soc. 159 (2015), no.2, 189-205.

\bibitem[Fox56]{Fox56}
R. H. Fox, \textit{Free differential calculus. {III}. {S}ubgroups}, Ann. of Math. (2) 64 (1956), 407-419. 


\bibitem[Hi01]{Hir} E. Hironaka, {\it Boundary manifolds of line arrangements}, Math. Ann. 319 (2001), no.1, 17-32.

\bibitem[Hir95]{H}  F. Hirzebruch, {\it  The topology of normal singularities of an algebraic surface (after D. Mumford)}, S\'{e}minaire Bourbaki, Vol. 8, Exp. No. 250, 129-137. Soci\'{e}t\'{e} Math\'{e}matique de France, Paris, 1995.

\bibitem[Lib82]{L82}
A. Libgober, \textit{Alexander polynomial of plane algebraic curves and cyclic multiple planes}, Duke Math. J.  49 (1982), no. 4, 833-851.

\bibitem[Lib94]{L94} A. Libgober, {\it Homotopy groups of the complements to singular hypersurfaces, $\uppercase\expandafter{\romannumeral 2}$},
Ann. of Math. (2)  139 (1994), no. 1, 117-144.

\bibitem[LL23]{LL23} Y. Liu, Y. Liu, {\it Integral homology groups of double coverings and rank one Z-local system for minimal CW complex}, Proc. Amer. Math. Soc. 151 (2023), no. 11, 5007-5012. 


\bibitem[Liu16]{Liu} Y. Liu, {\it Nearby cycles and Alexander modules of hypersurface complements}, Adv. Math.  291 (2016), 330-361. 


\bibitem[LM17]{LM} Y. Liu, L. Maxim, {\it Characteristic varieties of hypersurface complements}, Adv. Math. 306 (2017), 451-493.

\bibitem[LMW24]{LMW} Y. Liu, L. Maxim, B. Wang, {\it Cohomology of $\Z$-local systems on complex hyperplane arrangement complements}, Int. Math. Res. Not. IMRN 2024, no.~15, 11092-11103.


\bibitem[Max06]{Max}
L. Maxim, \textit{Intersection homology and Alexander modules of hypersurface complements}, Comment. Math. Helv. 81 (2006), no. 1, 123-155.

\bibitem[OT92]{OT} P. Orlik, M. Terao, {\it Arrangements of hyperplanes}, vol. 300, Springer-Verlag, 1992.



\bibitem[PS17]{PS17} S. Papadima, A. I. Suciu, {\it The Milnor fibration of a hyperplane arrangement: from modular resonance to algebraic monodromy}, Proc. Lond. Math. Soc. (3) 114 (2017), no. 6, 691-1004.


\bibitem[Ran02]{Ran02} R. Randell, {\it Morse theory, Milnor fibers and minimality of hyperplane arrangements}, Proc. Amer. Math. Soc. 130 (2002), no. 9, 2737-2743. 

\bibitem[Ran11]{Ran11} R. Randell, {\it The topology of hyperplane arrangements}, Topology of algebraic
varieties and singularities, Contemp. Math., vol. 538, Amer. Math. Soc., Providence, RI, 2011, pp. 309-318. 

\bibitem[Ryb11]{Ryb11}  G. L. Rybnikov, {\it  On the fundamental group of the complement of a complex hyperplane arrangement}, Funct. Anal. Appl. 45 (2011), no. 2, 137-148.
 
\bibitem[Suc01]{Suc01} A. I. Suciu, {\it Fundamental groups of line arrangements: Enumerative aspects} , Advances in algebraic geometry motivated by physics ({L}owell, {MA}, 2000), Contemp. Math., vol. 276, Amer. Math. Soc., Providence, RI, 2001, pp.~43-79.

\bibitem[Suc14]{Suc14}  A. I. Suciu, {\it Hyperplane arrangements and Milnor fibrations} , Ann. Fac. Sci. Toulouse Math. (6) 23 (2014), no.~2, 417-481.


\bibitem[Sug23]{Sug23} S. Sugawara, {\it $\Z$-local system cohomology of hyperplane arrangements and a Cohen-Dimca-Orlik type theorem}, Internat. J. Math. 34 (2023), no.~8, Paper No. 2350044, 15. 

\bibitem[Wes97]{W}  E. Westlund, {\it  The boundary manifold of an arrangement}, PhD. thesis. University of Wisconsin Madison (1997). 

\bibitem[Wil13]{Wil13} K. Williams, {\it The homology groups of the Milnor fiber associated to a central arrangement of hyperplanes in $\C^3$}, Topology Appl. 160 (2013) 1129-1143.


\bibitem[Yos20]{Yos20} M. Yoshinaga, {\it  Double coverings of arrangement complements and 2-torsion in Milnor fiber homology}, Eur. J. Math. 6 (2020), no. 3, 1097-1109.


\end{thebibliography}

\end{document}